\documentclass[a4paper,11pt]{article}

\usepackage{authblk}
\usepackage{fullpage}
\usepackage{graphicx,url}
\usepackage[utf8]{inputenc}
\usepackage[english]{babel}
\usepackage{amsmath,amsthm}
\usepackage{algorithm}
\usepackage[noend]{algpseudocode}

\newtheorem{theorem}{Theorem}

\newtheorem{lemma}[theorem]{Lemma}

\DeclareMathOperator{\mex}{mex}
\newcommand{\xor}{\oplus}

\title{The Normal Domination Partizan Game}

\author[1]{J. M. Brito}
\author[1]{Edileudo M. Moreira Filho}
\author[1]{Jefter G. M. Paz}
\author[1]{Rudini Sampaio}

\affil[1]{Dept. Computação, Universidade Federal do Ceará, Fortaleza, Brazil}

\date{}
\begin{document}

\maketitle

\begin{abstract}
The Domination game is an impartial game on graphs, introduced in 2010, and proved PSPACE-complete in the normal variant in 2026.
In this game, Alice and Bob alternately select playable vertices, where a vertex is playable if it dominates at least one vertex not dominated by the vertices selected before in the game. The game ends when the selected vertices form a dominating set. In the normal variant, the player unable to move loses. 
In contrast to the impartial game, the partizan game has the vertices already colored with $A$, $B$, or $C$, in such a way that Alice (resp. Bob) can only select vertices colored with $A$ (resp. $B$) or $C$.
The partizan game was proved PSPACE-hard in 2026.
In this paper, we determine the winner of the Normal Domination Partizan game in graphs whose connected components are complete bipartite graphs or complete split graphs, including star forests, for any initial coloring of its vertices.
\end{abstract}

\section{Introduction}

Given a graph $G$, we say that a subset $D\subseteq V(G)$ of vertices of $G$ is a \emph{dominating set} if every vertex has a neighbor in $D$ or belongs to $D$. We say that a vertex $u$ \emph{dominates} a vertex $v$ if $u$ and $v$ are neighbors or $u=v$.

In the Domination game on a graph $G$, Alice and Bob alternately select one playable vertex, where a vertex is \emph{playable} if it dominates at least one vertex not dominated by the vertices selected before in the game. The game ends when the set of selected vertices is a dominating set of $G$.

In the normal variant, the player unable to move loses. In the misère variant, the player unable to move wins. In the optimization variant, it is also given an integer $k$ and Alice wins if and only if at most $k$ vertices are selected. The smallest $k$ such that Alice wins this variant is the \emph{game domination number} $\gamma_g(G)$. It is not difficult to check that $\gamma(G)\leq\gamma_g(G)\leq 2\gamma(G)-1$, where the \emph{domination number} $\gamma(G)$ is the size of a minimum dominating set of $G$.

The Domination game was introduced in 2010 by Bre\v{s}ar et al. \cite{bresar10} in the optimization variant, which was proved PSPACE-complete in 2016 by Bre\v{s}ar et al. \cite{bresar16}. The normal variant was introduced by Brito et al. \cite{brito25}, where it was proved PSPACE-complete. They also proved that Alice wins the normal variant in a path $P_n$ (resp. cycle $C_n$) if and only if the remainder of the division of $n$ by $4$ is different from $0$ (resp. equal to $3$).

The Domination game is classified as an \emph{impartial game}, since there are no special moves for each player, that is, in any position of the game in a graph, if Alice can make certain move, Bob could also make it if it were his turn (and vice-versa). We say that a game is \emph{partizan} if it is not impartial.

In the Normal Domination Partizan game, the vertices of the graph $G$ are already colored with colors $A$, $B$ or $C$, in such a way that Alice (resp. Bob) can only select vertices colored with $A$ (resp. $B$) or $C$. Except for this difference, the game proceeds as in the impartial game. If all vertices are colored $C$, we are in the impartial game. Brito et al. \cite{brito25} also proved that the partizan game is PSPACE-complete and obtained a polynomial time algorithm to determine the winner of the partizan game in paths and cycles, for every initial coloring of the vertices with colors $A$ and $B$.

In this paper, we use the Combinatorial Game Theory to determine in polynomial time the winner of the Normal Domination Partizan game for every disjoint union of complete bipartite graphs and complete split graphs, which includes star forests, for any initial coloring of the vertices with colors $A$, $B$ and $C$.

\section{Theory of Normal Partizan Games in a Nutshell}

Books of Combinatorial Game Theory, such as \cite{rudini26,conway82,siegel13}, present the established theory on normal games, which associates to each position of a game a value from an abelian group created by Conway \cite{conway76} in the 1970s, containing rational numbers, but also different numbers such as $*$ (which is incomparable to $0$), $\uparrow$ and $\downarrow$.
It is known that a real number represents a finite combinatorial game if and only if  it is \emph{dyadic} (rational number whose denominator is a power of 2, including all integers).
Roughly speaking, let $s$ be the value of a game. It is known that $s>0$ (resp. $s<0$) if Alice (resp. Bob) always wins, playing first or not, that $s=0$ if the first player loses and that $s||0$ (incomparable to $0$) if the first player wins. 

The notation $\{X|Y\}$ describes the values associated to games, where $X$ (resp. $Y$) is the list of options (possible moves) of Alice (resp. Bob).
For example, in the empty game $0=\{|\}$, there is no option for Alice and Bob and then the first player loses.
In the game $1=\{0|\}$ (resp. $-1=\{|0\}$), Alice (resp. Bob) wins obtaining the game $0$ (the opponent has no move).
It is known that the game $1/2^{n+1}$ is equal to $\{0|1/2^n\}$ and that $-1/2^{n+1}=\{-1/2^n|0\}$.
In the game star $*=\{0|0\}$, the first player always wins, obtaining the game $0$, and then $*||0$.

The \emph{sum} of games $J_1+J_2$ is the game in which the next player must make a move in only one of the games $J_1$ or $J_2$. It is easy to see that the sum of games is commutative and associative.
Given a game $J$, let $nJ=n\cdot J$ be the sum of $n$ disjoint copies of $J$, let $nJ*=*+nJ$ and let $-J$ be the game in which the roles of Alice and Bob are changed.
It is known that $J+(-J)=0$ (the first player always loses in the sum of $J$ and $-J$).

Let $\uparrow=\{0|*\}$ and $\downarrow=\{*|0\}$.
The game $\uparrow^{[n]}$ is recursively defined as $\uparrow^{[0]}=0$, $\uparrow^{[1]}=\uparrow$ and $\uparrow^{[n]}=\{\uparrow^{[n-1]}|*\}$ for $n\geq 2$.
Analogously, $\downarrow_{[0]}=0$, $\downarrow_{[1]}=\downarrow$ and $\downarrow_{[n]}=\{*|\downarrow_{[n-1]}\}$ for $n\geq 2$.
It is known that $\uparrow^{[n]}=-\downarrow_{[n]}$, $\uparrow^{[n+1]}>\uparrow^{[n]}$, $\downarrow_{[n+1]}<\downarrow_{[n]}$, $\uparrow^{[n]}||*$ and $\downarrow_{[n]}||*$.

When $\{X|Y\}$ is impartial, then $X=Y$. In those games, in addition to $*0=0=\{|\}$ and $*1=*=\{0|0\}$, recursively define $*(n+1)=\{*0,*1,\ldots,*n\ |\ *0,*1,\ldots,*n\}$ for every natural $n$. The values $*n$ are called \emph{nimbers} for any natural $n$.
From the Sprague-Grundy Theory, it is known that $*m+*n=*(m\xor n)$, where $\xor$ is the bitwise xor operation. Moreover, if $X$ is a set of nimbers, then $\{X|X\}$ has value $\mex(X)$, which is the minimum nimber not belonging to $X$. For example, $*1+*2+*3=*(1\xor 2\xor 3)=*0=0$ and then the first player loses. Furthermore, if $X=\{*0,*1,*2,*5\}$, then $\{X|X\}=\mex(X)=*3$.

The following result can be found in the excellent book by Siegel \cite{siegel13}.

\begin{theorem}[\cite{siegel13}]\label{thm-basic1}
Let $n\geq 1$ be an integer. Then
\[
\left\{0\ \big|\ n\uparrow\right\} = \left\{n\uparrow \big|\ n\uparrow\right\} = (n+1)\uparrow*\ \ \ \ and\ \ \ \ \left\{0\ \big|\ n\uparrow*\right\} = \left\{n\uparrow*\ \big|\ n\uparrow*\right\} = (n+1)\uparrow
\]
\[
\left\{n\downarrow \big|\ 0\right\} = \left\{n\downarrow \big|\ n\downarrow\right\} = (n+1)\downarrow*\ \ \ \ and\ \ \ \ \left\{n\downarrow*\ \big|\ 0\right\} = \left\{n\downarrow*\ \big|\ n\downarrow*\right\} = (n+1)\downarrow
\]
\end{theorem}

The theorems below are of crucial importance to some proofs in the paper.

\begin{theorem}\label{thm-basic3}
Let $n\geq 0$. Then $\uparrow^{[n]}<2\uparrow$ and $\{\uparrow^{[n]}*,\uparrow|\uparrow\}=2\uparrow*$.
\end{theorem}

\begin{proof}
First we prove that $\uparrow^{[n]}-2\uparrow<0$, that is, Bob always wins in $J_1(n)=\uparrow^{[n]}+2\downarrow$.
If $n=0$, then $J_1(0)=0+2\downarrow<0$.
If $n=1$, then $J_1(1)=\uparrow+2\downarrow=\downarrow<0$.
So, assume that $n\geq 2$.
Recall that $\uparrow^{[n]}=\{\uparrow^{[n-1]}\ |\ *\}$ and $\downarrow=\{*|0\}$.
If Bob plays first in $J_1(n)$, then he obtains $*+2\downarrow<0$.
If Alice plays first in $J_1(n)$ obtaining $\uparrow^{[n-1]}+2\downarrow$, then Bob obtains $*+2\downarrow<0$.
If Alice plays first in $J_1(n)$ obtaining $\uparrow^{[n]}+\downarrow+*$, then Bob obtains $*+\downarrow+*<0$.

Now we prove that $J_2(n)=\{\uparrow^{[n]}*,\uparrow|\uparrow\}+2\downarrow*=0$, that is, the first player always loses in $J_2(n)$.
If $n=0$, then $J_2(0)=\{*,\uparrow|\uparrow\}+2\downarrow*=2\uparrow*+2\downarrow*=0$.
If $n=1$, then $J_2(1)=\{\uparrow*,\uparrow|\uparrow\}+2\downarrow*=2\uparrow*+2\downarrow*=0$.
So, assume that $n\geq 2$.
Recall that $\uparrow^{[n]}=\{\uparrow^{[n-1]}\ |\ *\}$, $\downarrow=\{*|0\}$ and $*=\{0|0\}$.

First assume that Alice plays first in $J_2(n)$.
If she obtains $\{\uparrow^{[n]}*,\uparrow|\uparrow\}+\downarrow*+*$, Bob obtains $\uparrow+\downarrow*+*=0$ and he wins.
If she obtains $\{\uparrow^{[n]}*,\uparrow|\uparrow\}+2\downarrow+0$, Bob obtains $\uparrow+2\downarrow<0$ and he wins.
If she obtains $\uparrow+2\downarrow*$, Bob obtains $\uparrow+2\downarrow+0<0$ and he wins.
If she obtains $\uparrow^{[n]}*+2\downarrow*$, this game is exactly $J_1(n)<0$ and he wins.

Now assume that Bob plays first in $J_2(n)$.
If he obtains $\uparrow+2\downarrow*$, Alice obtains $\uparrow+\downarrow*+*=0$ and she wins.
If he obtains $\{\uparrow^{[n]}*,\uparrow|\uparrow\}+\downarrow*+0$, Alice obtains $\{\uparrow^{[n]}*,\uparrow|\uparrow\}+*+*>0$ and she wins.
If he obtains $\{\uparrow^{[n]}*,\uparrow|\uparrow\}+2\downarrow+0$, Alice obtains $\{\uparrow^{[n]}*,\uparrow|\uparrow\}+\downarrow*$ and Bob has three options.
If he obtains $\uparrow+\downarrow+*$, Alice obtains $\uparrow+\downarrow+0=0$ and wins.
If he obtains $\{\uparrow^{[n]}*,\uparrow|\uparrow\}+\downarrow+0$, Alice obtains $\uparrow+\downarrow=0$ and wins.
If he obtains $\{\uparrow^{[n]}*,\uparrow|\uparrow\}+0+*$, Alice obtains $\{\uparrow^{[n]}*,\uparrow|\uparrow\}+0+0>0$ and wins.
\end{proof}

\begin{theorem}\label{thm-basic2}
For every integer $n\geq 1$:
$$\uparrow^{[n]}*\ =\ \left\{0, \uparrow^{[n-1]}*\ \big|\ 0, \uparrow^{[n+1]}*\right\}\ =\ \left\{0,  \uparrow^{[n-1]}*\ \big|\ 0\right\}$$
$$\downarrow_{[n]}*\ =\ \left\{0,\downarrow_{[n+1]}*\ \big|\ 0, \downarrow_{[n-1]}*\right\}\ =\ \left\{0\ \big|\ 0, \downarrow_{[n-1]}*\right\}$$
Furthermore,
$$\uparrow^{[n]}*\ =\ \left\{0,\{0,\uparrow^{[n]}*\ |\ 0,\uparrow^{[n]}*\}\ \ \ \Big|\ \ \ 0,\{0,\uparrow^{[n]}*\ |\ 0,\uparrow^{[n]}*\}\right\}$$
$$\downarrow_{[n]}*\ =\ \left\{0,\{0,\downarrow_{[n]}*\ |\ 0,\downarrow_{[n]}*\}\ \ \ \Big|\ \ \ 0,\{0,\downarrow_{[n]}*\ |\ 0,\downarrow_{[n]}*\}\right\}$$
\end{theorem}

\begin{proof}
We first prove that $J_1=\{0,\uparrow^{[n-1]}*|0\}+\downarrow_{[n]}*=0$, that is, the first to play in $J_1$ loses.
Suppose that Alice plays first in $J_1$.
If she obtains $0+\downarrow_{[n]}*$, Bob obtains $\downarrow_{[n]}<0$ and wins.
If she obtains $\uparrow^{[n-1]}*+\downarrow_{[n]}*<0$, Bob wins.
If she obtains $\{0,\uparrow^{[n-1]}*|0\}+\downarrow_{[n]}+0$ (since $*=\{0|0\}$), Bob obtains $0+\downarrow_{[n]}<0$ and wins.
If she obtains $\{0,\uparrow^{[n-1]}*|0\}+*+*$ (since $\downarrow_{[n]}=\{*|\downarrow_{[n-1]}\}$), Bob obtains $0+*+*=0$ and wins.
Now suppose that Bob plays first in $J_1$.
If he obtains $0+\downarrow_{[n]}*$, Alice obtains $0+*+*=0$ and wins.
If he obtains $\{0,\uparrow^{[n-1]}*|0\}+\downarrow_{[n-1]}+*$ (since $\downarrow_{[n]}=\{*|\downarrow_{[n-1]}\}$), Alice obtains $\uparrow^{[n-1]}*+\downarrow_{[n-1]}*=0$ and wins.
If he obtains $\{0,\uparrow^{[n-1]}*|0\}+\downarrow_{[n]}+0$, Alice obtains $\{0,\uparrow^{[n-1]}*|0\}+*$ and there are 2 possibilities for the second move of Bob: (i) if Bob obtains $0+*$, Alice wins; and (ii) if Bob obtains $\{0,\uparrow^{[n-1]}*|0\}+0$, Alice obtains $0$ and wins. Since the second player always wins, $J_1=0$.

Now let us prove that $J_2=\{0,\uparrow^{[n-1]}*|0,\uparrow^{[n+1]}*\}+\downarrow_{[n]}*=0$, that is, the first to play in $J_2$ loses.
The only difference to the last case is when Bob plays first, obtaining $\uparrow^{[n+1]}*+\downarrow_{[n]}*>0$, which is winner for Alice. Since the second player always wins, $J_2=0$.

Finally, let us prove that $J_3 = \Delta+\downarrow_{[n]}*=0$, that is, the first to play in $J_3$ loses, where
\[
\Delta\ =\ \left\{0,\{0,\uparrow^{[n]}*\ |\ 0,\uparrow^{[n]}*\}\ \ \ \Big|\ \ \ 0,\{0,\uparrow^{[n]}*\ |\ 0,\uparrow^{[n]}*\}\right\}.
\]
Suppose that Alice plays first. If she obtains $0+\downarrow_{[n]}*$, Bob obtains $\downarrow_{[n]}<0$ and wins.
If she obtains $\{0,\uparrow^{[n]}*\ |\ 0,\uparrow^{[n]}*\}+\downarrow_{[n]}*$, Bob obtains $\uparrow^{[n]}*+\downarrow_{[n]}*=0$ and wins.
If she obtains $\Delta+\downarrow_{[n]}$, Bob obtains $0+\downarrow_{[n]}<0$ and wins.
If she obtains $\Delta+*+*$, Bob obtains $0+*+*=0$ and wins.
Now suppose that Bob is the first to play in $J_3$.
If he obtains $0+\downarrow_{[n]}+*$, Alice obtains $0+*+*=0$ and wins.
If he obtains $\{0,\uparrow^{[n]}*\ |\ 0,\uparrow^{[n]}*\}+\downarrow_{[n]}*$, Alice obtains $\uparrow^{[n]}*+\downarrow_{[n]}*=0$ and wins.
If he obtains $\Delta+\downarrow_{[n]}$, Alice obtains $\Delta+*$ and  there are 3 possibilities for the second move of Bob: (i) if Bob obtains $\Delta+0$, Alice obtains $0$ and wins; (ii) if Bob obtains $0+*$, Alice wins; and (iii) if Bob obtains $\{0,\uparrow^{[n]}*\ |\ 0,\uparrow^{[n]}*\}+*$, Alice obtains $\uparrow^{[n]}*+*>0$ and wins.
Finally, if in the first move he obtains $\Delta+\downarrow_{[n-1]}*$, Alice obtains $\{0,\uparrow^{[n]}*\ |\ 0,\uparrow^{[n]}*\}+\downarrow_{[n-1]}*$ and there are 4 possibilities for the second move of Bob: (i) if Bob obtains $0+\downarrow_{[n-1]}*$, Alice obtains $0+*+*=0$ and wins; (ii) if Bob obtains $\uparrow^{[n]}*+\downarrow_{[n-1]}*>0$, Alice wins; (iii) if Bob obtains $\{0,\uparrow^{[n]}*|0,\uparrow^{[n]}*\}+\downarrow_{[n-1]}+0$, Alice obtains $\{0,\uparrow^{[n]}*|0,\uparrow^{[n]}*\}+*>0$ and wins; and (iv) if Bob obtains $\{0,\uparrow^{[n]}*|0,\uparrow^{[n]}*\}+\downarrow_{[n-2]}*$, Alice obtains $\uparrow^{[n]}*+\downarrow_{[n-2]}*>0$ and wins.
Since the second player always wins, $J_3=0$.
\end{proof}

In the rest of the paper, we use the following notation, where $k$ is a natural number. Let
\[
\Delta(k)=
\begin{cases}
    *,&\mbox{if $k$ is odd},\\
    0,&\mbox{if $k$ is even}
\end{cases}
\]
Also let $\Gamma(k)=(k-1)\uparrow+\Delta(k)$.
Notice that Theorem \ref{thm-basic1} implies that
\[
\{0\ |\ \Gamma(k)\}\ =\ \{\Gamma(k)\ |\ \Gamma(k)\}\ =\ \Gamma(k+1).
\]

\section{Complete Split Graphs}\label{sec-split}

In this section, we show that the game in a complete split graph can be reduced to a game in a star, which is solved in Section \ref{sec-star}.
A \emph{complete split graph} is any graph $G$ such that $V(G)$ has a partition $(K,S)$ such that $K$ induces a clique, $S$ induces an independent set and there are all possible edges between $K$ and $S$.
Notice that complete graphs are complete split graphs with $S=\emptyset$.

We first prove a theorem on general graphs with true twins, where we say that two vertices are \emph{true twins} if they are adjacent and have the same neighbors

\begin{theorem}\label{thm-twins}
Let $G$ be a graph in which every vertex is colored $A$, $B$ or $C$. Let $u$ and $v$ be true twins of $G$.
Let $G'$ be the graph obtained from $G$ by removing $u$ and recoloring $v$ with color $C$ if $u$ and $v$ have different colors in $G$. Then the Normal Domination Partizan game in $G$ is equivalent to the game in $G'$.
\end{theorem}

\begin{proof}
Notice that, in the game in $G$, if either $u$ or $v$ is played, then the other becomes unplayable. If both have color $A$, only Alice can play them. If both have color $B$, only Bob can play them. Otherwise, both players can make one move in either $u$ or $v$. This is equivalent to the game in $G'$.
\end{proof}

With this, the result on complete graphs and complete split graphs are direct consequences.

\begin{theorem}\label{thm-clique}
The value of a complete graph in the Domination game is $1$ if every vertex has color $A$, or it is $-1$ if every vertex has color $B$, or it is $*$, otherwise.
\end{theorem}

\begin{proof}
Notice that all vertices of a clique are true twins. Then applying Theorem \ref{thm-twins} successively, we obtain a graph $G'$ with only one vertex $v$, whose color is $A$ (resp. $B$) if every vertex of the clique has color $A$ (resp. $B$), or is $C$, otherwise.
The game in $G'$ is $\{0|\}=1$ if $v$ has color $A$ in $G'$, since Alice always wins in the first move (leading to a game $0$) and Bob has no move.
The game in $G'$ is $\{|0\}=-1$ if $v$ has color $B$ in $G'$, since Bob always wins in the first move (leading to a game $0$) and Alice has no move.
Otherwise, the first player (Alice or Bob) always wins in the first move (leading to a game $0$) and then the game in $G'$ is $\{0|0\}=*$.
\end{proof}

\begin{theorem}\label{thm-split}
Consider the game in a complete split graph $G$ with partition $(K,S)$. If every vertex of the clique $K$ has color $A$ (resp. $B$), the game is equivalent to the game in the star with center colored $A$ (resp. $B$) and leaves in $S$. Otherwise, the game is equivalent to the game in the star  with center colored $C$ and leaves in $S$.
\end{theorem}

\begin{proof}
Notice that all vertices of $K$ are true twins. Then, we have the result applying Theorem \ref{thm-twins} successively on the vertices of $K$.
\end{proof}

\section{Stars}\label{sec-star}

In this section, we solve the game for any initial coloring of a star $K_{1,n}$ in Theorems \ref{thm-star1}, \ref{thm-star2}, \ref{thm-star3}, \ref{thm-star4} and \ref{thm-star5}. This implies the solution of the game in complete split graphs from Theorem \ref{thm-split}.

\begin{theorem}\label{thm-star1}
Consider the game in a star $K_{1,n}$.
If every vertex has color $C$, the value is $*$ if $n$ is odd, or it is $*2$ if $n$ is even.
\end{theorem}

\begin{proof}
By induction on $n$.
Note that $K_{1,1}$ has value $\mex\{0\}=*$.
Let $n>1$ and suppose it holds true for values less than $n$.
Thus $K_{1,n}$ has value $\mex\{0,K_{1,n-1}\}$. If $n$ is odd, $K_{1,n-1}$ has value $*2$ by induction and then $K_{1,n}$ has value $\mex\{0,*2\}=*$. If $n$ is even, $K_{1,n-1}$ has value $*$ by induction and then $K_{1,n}$ has value $\mex\{0,*\}=*2$. 
\end{proof}

\begin{theorem}\label{thm-star2}
Consider the game in a star $K_{1,n}$ with $a$ leaves of color $A$, $b=n-a$ leaves of color $B$ and $n\geq 2$ such that the universal vertex $u$ has color $A$.
Then the value is $a-b$ if $b<a$, or it is $1/2^{b-a+1}$ if $b\geq a\geq 1$, or it is $\Gamma(b)$ if $a=0$.
In the case of universal vertex with color $B$, the value is $-J$, where $J$ is the value changing the colors $A$ and $B$.
\end{theorem}

\begin{proof}
By induction on $n=a+b$. Suppose that the color of the universal vertex $u$ is $A$ (the case of color $B$ is analogous). Let us represent the game as $(a,b)$. If $(a,b)=(2,0)$, then $(a,b)=\{0,1|\}=2$. If $(a,b)=(1,1)$, then $(a,b)=\{0,*|1\}=1/2$. If $(a,b)=(0,2)$, then $(a,b)=\{0|*\}=\uparrow$.
Let $n\geq 3$ and suppose it holds true for values less than $n$. Let us prove for $n=a+b$. Note that, by induction and  by Theorem \ref{thm-basic1}, where $k\geq 2$ is an integer:
\begin{itemize}
\item $(0,2k-1)\ =\ \{0\ \big|\ (0,2k-2)\}\ =\ \{0\ \big|\ (2k-3)\uparrow\}\ =\ (2k-2)\uparrow*$
\item $(0,2k)\ =\ \{0\ \big|\ (0,2k-1)\}\ =\ \{0\ \big|\ (2k-2)\uparrow*\}\ =\ (2k-1)\uparrow$
\item $(1,2k-1)\ =\ \{0, (0,2k-1)\ \big|\ (1,2k-2)\}\ =\ \{(2k-1)\uparrow*\ \big|\ 1/2^{2k-2}\}\ =\ 1/2^{2k-1}$
\item $(1,2k)\ =\ \{0, (0,2k)\ \big|\ (1,2k-1)\}\ =\ \{(2k-1)\uparrow\ \big|\ 1/2^{2k-1}\}\ =\ 1/2^{2k}$
\end{itemize}
Thus suppose that $a\geq 2$. Note that
\begin{itemize}
\item $(a,0)\ =\ \{0,(a-1,0)|\}\ =\ \{a-1|\}\ =\ a$
\item $(a,a-1)\ =\ \{0,(a-1,a-1)\ \big|\ (a,a-2)\}\ =\ \{1/2\ \big|\ 2\}\ =\ 1$
\item $(a,a)\ \ \ \ \ \ \ \ =\ \{0,(a-1,a)\ \big|\ (a,a-1)\}\ =\ \{1/4\ \big|\ 1\}\ =\ 1/2$
\item $(a,a+1)\ =\ \{0,(a-1,a+1)\ \big|\ (a,a)\}\ =\ \{1/8\ \big|\ 1/2\}\ =\ 1/4$
\end{itemize}
If $b<a-1$ and $b>0$, then
\begin{itemize}
\item $(a,b)\ =\ \{0,(a-1,b)\ \big|\ (a,b-1)\}\ =\ \{a-b-1\ \big|\ a-b+1\}\ =\ a-b$
\end{itemize}
If $b>a+1$, then
\begin{itemize}
\item $(a,b)\ =\ \{0,(a-1,b)\ \big|\ (a,b-1)\}\ =\ \{1/2^{b-a+2}\ \big|\ 1/2^{b-a}\}\ =\ 1/2^{b-a+1}$
\end{itemize}
\end{proof}

\begin{theorem}\label{thm-star3}
Consider the game in a star $K_{1,n}$ with $a$ leaves with color $A$, $b=n-a$ leaves of color $B$ and $n\geq 1$ such that the universal vertex $u$ has color $C$.
If $a=b$, the value is $*2$. If $|a-b|=1$, the value is $*$.
If $a\geq b+2$, the value is $\uparrow^{[a-b-1]}*$.
If $b\geq a+2$, the value is $\downarrow_{[b-a-1]}*$.
\end{theorem}

\begin{proof}
By induction on $n=a+b$. Let us represent the game as $(a,b)$.
If $(a,b)$ is $(1,0)$ or $(0,1)$, the value is $\{0|0\}=*$.
If $(a,b)$ is $(1,1)$, the value is $\{0,*|0,*\}=*2$.
If $(a,b)$ is $(2,0)$, the value is $\{0,*|*\}=\uparrow*$.
If $(a,b)$ is $(0,2)$, the value is $\{0|0,*\}=\downarrow*$.
So let $n\geq 3$ and suppose it holds true for values less than $n$. Let us prove for $n=a+b$.
Note that
\begin{itemize}
\item $(a,0)\ =\ \left\{0,(a-1,0)\ \big|\ 0\right\}\ =\ \left\{0,\uparrow^{[a-2]}*\ \big|\ 0\right\}\ =\ \uparrow^{[a-1]}*$
\item $(0,b)\ =\ \left\{0\ \big|\ 0,(0,b-1)\right\}\ =\ \left\{0\ \big|\ 0,\downarrow_{[b-2]}*\right\}\ =\ \downarrow_{[b-1]}*$
\end{itemize}
by induction and by Theorem \ref{thm-basic2}.
Thus assume that $a,b\geq 1$ and $a+b\geq 2$.
Note that:
\begin{itemize}
\item $(a,a-2)\ =\ \left\{0,(a-1,a-2)\ \big|\ 0,(a,a-3)\right\}\ =\ \left\{0,*\ \big|\ 0,\uparrow^{[2]}*\right\}\ =\ \uparrow*$
\item $(a,a-1)\ =\ \left\{0,(a-1,a-1)\ \big|\ 0,(a,a-2)\right\}\ =\ \left\{0,*2\ \big|\ 0,\uparrow*\right\}\ =\ *$
\item $(a,a)\ =\ \left\{0,(a-1,a)\ \big|\ 0,(a,a-1)\right\}\ =\ \left\{0,*\ \big|\ 0,*\right\}\ =\ *2$
\item $(a,a+1)\ =\ \left\{0,(a-1,a+1)\ \big|\ 0,(a,a)\right\}\ =\ \left\{0,\downarrow*\ \big|\ 0,*2\right\}\ =\ *$
\item $(a,a+2)\ =\ \left\{0,(a-1,a+2)\ \big|\ 0,(a,a+1)\right\}\ =\ \left\{0,\downarrow_{[2]}*\ \big|\ 0,*\right\}\ =\ \downarrow*$
\end{itemize}
If $a\geq b+3$, then
\begin{itemize}
\item $(a,b) = \left\{0,(a-1,b)\ \big|\ 0,(a,b-1)\right\} = \left\{0,\uparrow^{[a-b-2]}*\ \big|\ 0,\uparrow^{[a-b]}*\right\}\ =\ \uparrow^{[a-b-1]}*$
\end{itemize}
If $b\geq a+3$, then
\begin{itemize}
\item $(a,b) = \left\{0,(a-1,b)\ \big|\ 0,(a,b-1)\right\} = \left\{0,\downarrow_{[b-a-2]}*\ \big|\ 0,\downarrow_{[b-a]}*\right\}\ =\ \downarrow_{[b-a-1]}*$
\end{itemize}
\end{proof}

\begin{theorem}\label{thm-star4}
Consider the game is a star $K_{1,n}$ with $a$ leaves with color $A$, $b$ leaves with color $B$ and $c=n-a-b\geq 1$ leaves with color $C$ such that the universal vertex has color $A$.
If $a=0$, then the value is $\Gamma(b+c)$. If $a>0$, the value is $J$ if $c$ is even and it is $J*$ if $c$ is odd, where $J$ is the value removing the leaves of color $C$.
In the case of universal vertex with color $B$, the value is $-J$, where $J$ is the value changing the colors $A$ and $B$.
\end{theorem}

\begin{proof}
By induction on $n=a+b+c$. Let us represent the game as $(a,b,c)$.
If $(a,b,c)$ is $(0,0,1)$, then the value is $\{0|0\}=*$.
If $(a,b,c)=(1,0,1)$, the value is $\{1|1\}=1*$.
If $(a,b,c)$ is $(0,1,1)$ or $(0,0,2)$, the value is $\{0,*|*\}=\uparrow$.
Thus let $n\geq 3$ and suppose it holds true for values less than $n$. Let us prove for $n=a+b+c$.
Note that $(0,b,c)=\left\{0,(0,b,c-1)\ \big|\ (0,b-1,c),(0,b,c-1)\right\}$. 
If $b+c$ is even:
\begin{itemize}
\item $(0,b,c)\ =\ \left\{0,(n-2)\uparrow *\ \big|\ (n-2)\uparrow *\right\}\ =\ (n-1)\uparrow$.
\end{itemize}
If $b+c$ is odd, then:
\begin{itemize}
\item $(0,b,c)\ =\ \left\{0,(n-2)\uparrow\ \big|\ (n-2)\uparrow\right\}\ =\ (n-1)\uparrow *$.
\end{itemize}
Thus suppose that $a>0$. Note that $(a,b,c)\ =\ \{0,(a-1,b,c),(a,b,c-1)\ \big|\ (a,b-1,c),(a,b,c-1)\}$ and that $(a,b,c-1)>(a,b,c-1)>0$ and $(a,b,c-1)>(a-1,b,c)$. Therefore, $(a,b,c)=\{(a,b,c-1)\ |\ (a,b,c-1)\}$. If $c$ is even, then
\begin{itemize}
\item $(a,b,c)\ =\ \left\{(a,b)*\ \big|\ (a,b)*\right\}\ =\ (a,b)$,
\end{itemize}
since $(a,b)>1/2^x$ for some integer $x$ by Theorem \ref{thm-star2}.
Similarly, if $c$ is odd:
\begin{itemize}
\item $(a,b,c)\ =\ \left\{(a,b)\ \big|\ (a,b)\right\}\ =\ (a,b)*$.
\end{itemize}

\end{proof}

\begin{theorem}\label{thm-star5}
Consider the game in a star $K_{1,n}$ with $a$ leaves with color $A$, $b$ leaves with color $B$ and $c=n-a-b\geq 1$ leaves with color $C$ such that the universal vertex has color $C$ and $a+b\geq 1$.
If $c$ is even, the value is the same removing the leaves with color $C$.
So assume $c$ is odd.
If $a=b$, the value is $*$. If $|a-b|=1$, the value is $*2$.
If $a\geq b+2$, the value is $\left\{0,\uparrow^{[a-b-1]}*\ \big|\ 0,\uparrow^{[a-b-1]}*\right\}$.
If $b\geq a+2$, the value is $\left\{0,\downarrow_{[b-a-1]}*\ \big|\ 0,\downarrow_{[b-a-1]}*\right\}$.
\end{theorem}

\begin{proof}
By induction on $n=a+b+c$. Let us represent the game as $(a,b,c)$.
If $(a,b,c)$ is $(1,0,1)$ or $(0,1,1)$, the value is $\{0,*|0,*\}=*2$.
If $(a,b,c)$ is $(1,1,1)$ or $(1,0,2)$ or $(0,1,2)$, the value is $\{0,*2|0,*2\}=*$.
If $(a,b,c)$ is $(2,0,1)$, the value is $\{0,\uparrow*,*2|0,\uparrow*\}=\{0,\uparrow*|0,\uparrow*\}$.
If $(a,b,c)$ is $(0,2,1)$, the value is $\{0,\downarrow*,*2|0,\downarrow*\}=\{0,\downarrow*|0,\downarrow*\}$.
Thus let $n\geq 4$ and suppose it holds true for values less than $n$. Let us prove for $n=a+b+c$.
Note that $(a,b,c)=\{0,(a-1,b,c),(a,b,c-1)\ |\ 0,(a,b-1,c),(a,b,c-1)\}$ and that $(a-1,b,c)<(a,b,c-1)<(a,b-1,c)$.
Therefore, if $c$ is odd, then
\begin{itemize}
\item $(a,b,c)\ =\ \left\{0,(a,b)\ \big|\ 0,(a,b)\right\}$
\end{itemize}
and the result follows from Theorem \ref{thm-star3}.
Thus suppose that $c$ is even. Therefore:
$$(a,b,c)\ =\ \left\{0,\ \{0,(a,b)\ |\ 0,(a,b)\}\ \ \ \Big|\ \ \ 0,\{0,(a,b)\ |\ 0,(a,b)\}\right\}.$$
If $a=b$, then by Theorem \ref{thm-star3}
\begin{itemize}
\item $(a,b,c)\ =\ \left\{0,\{0,*2|0,*2\}\ \big|\ 0,\{0,*2|0,*2\}\right\}\ =\ \{0,*|0,*\}\ =\ *2$.
\end{itemize}
If $|a-b|=1$, then by Theorem \ref{thm-star3}
\begin{itemize}
\item $(a,b,c)\ =\ \left\{0,\{0,*|0,*\}\ \big|\ 0,\{0,*|0,*\}\right\}\ =\ \{0,*2|0,*2\}\ =\ *$.
\end{itemize}
Otherwise, by Theorem \ref{thm-star3}, $(a,b)$ is $\uparrow^{[a-b-1]}*$ or $\downarrow_{[a-b-1]}*$.
Therefore, by Theorem \ref{thm-basic2}, $(a,b,c) = (a,b)$.
\end{proof}

\section{Complete Bipartite Graphs}\label{sec-bip}

Notice that stars are complete split graphs, which are solved in Section \ref{sec-split}, but they are also complete bipartite graphs.
In this section, we solve the remaining cases of complete bipartite graphs, that is, the cases in which both parts have at least two vertices.

The next theorem deals with some initial colorings of complete bipartite graphs.

\begin{theorem}\label{thm-bip1}
Consider the Domination game in the complete bipartite graph $K_{s,t}$ with parts $S$ and $T$ with sizes $|S|=s\geq 2$ and $|T|=t\geq 2$.
If every vertex of $S\cup T$ has color $A$ (resp. $B$), the value is $\max\{s,t\}$ (resp. $-\max\{s,t\}$).
If $S$ and $T$ have at least (i) one vertex of color $C$ or (ii) one vertex of color $A$ and one vertex of color $B$, then the value is $0$. If every vertex of $S$ has color $A$ (resp. $B$) and every vertex of $T$ has color $B$ (resp. $A$), then the value is $0$.
\end{theorem}

\begin{proof}
Let us represent the game as $(s,t)$.
Also let $(s,t]$ be the game in which the vertices of $T$ do not need to be dominated (but they can be selected in order to dominate the vertices of $S$). Analogously for $[s,t)$.
First suppose that every vertex has color $A$.
Note that $(s,t]=\{0,(s-1,t]\ |\ \}$ for $s,t\geq 1$ and that $(0,t]=0$. Therefore, it is easy to check that $(s,t]=s$ and $[s,t)=t$ for every $s,t\geq 1$.
Let us prove by induction in $n=s+t$ the value of $(s,t)$. Note that $(n,0)$ and $(0,n)$ has value $n$, since every vertex must be dominated by itself. Let $n>1$ and assume it holds true for values less than $n$. Let us prove the result for $n=s+t$:
\begin{itemize}
\item $(s,t)\ =\ \{(s-1,t],[s,t-1)\ |\ \}\ =\ \max(s-1,t-1)+1\ =\ \max(s,t)$
\end{itemize}

Now suppose that $S$ and $T$ satisfy (i) or (ii).
If the first player selects a vertex of $S$ (resp. $T$), then the other player can select a vertex of $T$ (resp. $S$), wining the game, which therefore has value $0$.

Finally, if every vertex of $S$ has color $A$ (resp. $B$) and every vertex of $T$ has color $B$ (resp. $A$), then the second player can always make the first move in the part distinct from the part played by the first player, winning the game, which therefore has value $0$.
\end{proof}

The case of complete bipartite graphs in which both parts have at least two vertices, one part contains only one color and this color is $A$ or $B$ and the other part contains color $C$ or at least two colors is more complicated.

From now on, without loss of generality, we assume that $G$ is a complete bipartite graph with parts $S$ and $T$ such that all vertices of $S$ are colored $A$ and $T$ contains the color $C$ or at least two colors. Let $s=|S|$ and $t=|T|$. Also let $a$, $b$ and $c$ be the number of vertices in $T$ colored $A$, $B$ and $C$, respectively. Notice that $a+c\geq 1$ and $b+c\geq 1$, that is, Alice and Bob can always play in $T$ as first player.
We represent the game in $G$ as $(s,a,b,c)$.
Also let $[s,a,b,c)$ (resp. $(s,a,b,c]$) be the game in which the vertices of $S$ (resp. $T$) do not need to be dominated, but they can be selected in order to dominate the vertices of $T$ (resp. $S$).

\begin{lemma}\label{lem-esq}
Let $a$, $b$ and $c$ such that $a+c\geq 1$ and $b+c\geq 1$.
If $s\geq 2$, then $(s,a,b,c]\ =\ \uparrow^{[s-1]}*$. Moreover, $(1,a,b,c]=*$ and $(0,a,b,c]=0$.
\end{lemma}

\begin{proof}
Notice that $(0,a,b,c]=0$, since every vertex is dominated, and $(1,a,b,c]=\{0|0\}=*$, since the first player (Alice or Bob) wins the game in the first move.
So, let $s\geq 2$.
Then $(s,a,b,c] = \{(s-1,a,b,c],0\ |\ 0\}$, since Alice and Bob can play in $T$ as first player, ending the game, and Alice can play in $S$, obtaining the game $(s-1,a,b,c]$.
Since $\uparrow^{[n]}*=\{\uparrow^{[n-1]}*,0|0\}$ by Theorem \ref{thm-basic2} for every $n\geq 1$, we have the result by induction.
\end{proof}

\begin{lemma}\label{lem-czero}
Let $s\geq 2$ and $a,b\geq 0$. Then
\[
  [s,a,b,0)\ =\ 
 \begin{cases}
    a-b, &\mbox{if $a>b$},\\
    1/2^{b-a+1}, &\mbox{if $b\geq a>0$},\\
    \Gamma(b), &\mbox{if $a=0$ and $b>0$},\\
    0, &\mbox{if $a=b=0$}.
 \end{cases}
\]
\end{lemma}

\begin{proof}
Notice that $[s,0,0,0)=0$, since all vertices are dominated.
Also $[s,1,0,0)=\{0|\}=1$, since Alice wins in her first move and Bob has no move, and $[s,0,1,0)=\{0|0\}=*$, since the first player wins in the first move.
Furthermore, $[s,1,1,0)=\{\ 0,[s,0,1,0)\ |\ [s,1,0,0)\}=\{0,*|1\}=1/2$.
Also note that, if $a\geq 2$ then by induction $[s,a,0,0)=\{\ [s,a-1,0,0),0\ |\ \ \}=\{a-1|\ \}=a$.
Moreover, if $b\geq 2$, then, by induction and by Theorem \ref{thm-basic1}, $[s,0,b,0)=\{0\ |\ [s,0,b-1,0)\}=\{0\ |\ \Gamma(b-1)\}=\Gamma(b)$.
Thus, from now on, we may assume that $a,b\geq 1$ and $a+b\geq 3$.

With this, notice that
$[s,a,b,0)\ =\ \{0,[s,a-1,b,0)\ |\ [s,a,b-1,0)\}$.
Since $[s,a-1,b,0)>0$ by induction, then
$[s,a,b,0)\ =\ \{\ [s,a-1,b,0)\ |\ [s,a,b-1,0)\}$.
Therefore, if $a=1$, then by induction
$[s,1,b,0)=\{\ [s,0,b,0)\ |\ [s,1,b-1,0)\}=\{\Gamma(b)\ |\ 1/2^{b-1}\}=1/2^b=1/2^{b-a+1}$, since $a+b\geq 3$.
So, assume that $a\geq 2$.
If $a>b+1$, then
$[s,a,b,0)=\{a-1-b\ |\ a-b+1\}=a-b$.
If $a=b+1$, then
$[s,a,b,0)=\{1/2\ |\ 2\}=1=a-b$.
If $a=b$, then
$[s,a,b,0)=\{1/4\ |\ 1\}=1/2=1/2^{b-a+1}$.
If $a<b-1$, then
$[s,a,b,0)=\{1/2^{b-a+2}\ |\ 1/2^{b-a}\}=1/2^{b-a+1}$.
\end{proof}

We say that a value $x$ is \emph{infinitesimal} if $-1/2^n<x<1/2^n$ for every natural $n$. It is know \cite{siegel13} that $0$, $*$, $m\cdot\uparrow^{[n]}$ and $m\cdot\downarrow_{[n]}$ are infinitesimal for every $m,n\geq 1$.
The next lemma shows that $[s,a,b,0)-[s,x,y,0)$ is positive and not infinitesimal if $a\geq 1$ and $a-b>x-y$.

\begin{lemma}\label{lem-czero2}
Let $a\geq 1$ and $b,x,y\geq 0$ be integers such that $a-b>x-y$. Then $[s,a,b,0)\geq[s,x,y,0)+1/2^{a-b+2}$.
\end{lemma}

\begin{proof}
First assume that $a>b$. Then $[s,a,b,0)=a-b$ by Lemma \ref{lem-czero}. Then the possible values of $[s,x,y,0)$ are $x-y$, $1/2^{y-x+1}$, $\Gamma(y)$ and $0$, depending on the cases of Lemma \ref{lem-czero}. All this values differ by at least 1 from $a-b$, since $a-b>x-y$.

Now suppose that $b\geq a>0$ and consequently $y\geq x$. Then $[s,a,b,0)=1/2^{b-a+1}$. As before, the possible values of $[s,x,y,0)$ are $1/2^{y-x+1}$, $\Gamma(y)$ and $0$, depending on the cases of Lemma \ref{lem-czero}. All this values differ by at least $1/2^{b-a+2}$ from $1/2^{b-a+1}$.
\end{proof}

\begin{lemma}\label{lem-azero}
Let $s\geq 2$ and $c\geq 1$. Then $[s,0,b,c)=\Gamma(b+c)$.
\end{lemma}

\begin{proof}
Recall from Lemma \ref{lem-czero} that $[s,0,b,0)=\Gamma(b)$ for $b\geq 1$.
First assume that $b=0$. 
If $c=1$, then $[s,0,0,1)=\{0|0\}=*=\Gamma(1)$, since the first player always wins in the first move.
So, suppose that $c\geq 2$.
Then $[s,0,0,c)=\{0,[s,0,0,c-1)\ |\ [s,0,0,c-1)\}=\{0,\Gamma(c-1)\ |\ \Gamma(c-1)\}=\Gamma(c)$ by induction and by Theorem \ref{thm-basic1}.

Now assume that $b\geq 1$. 
Then, $[s,0,b,c)\ =\ \{0,[s,0,b,c-1)\ |\ [s,0,b-1,c),[s,0,b,c-1)\}$. Thus, by induction and by Theorem \ref{thm-basic1}, $[s,0,b,c)\ =\ \{0,\Gamma(b+c-1)|\Gamma(b+c-1)\}\ =\ \Gamma(b+c)$.

\end{proof}

\begin{lemma}\label{lem-cdifz}
Let $s\geq 2$, $a\geq 1$ and $b,c\geq 0$. Then $[s,a,b,c)\ =\ [s,a,b,0)+\Delta(c)$.
\end{lemma}

\begin{proof}
Notice that $[s,1,0,c)=\{0,[s,1,0,c),[s,1,0,c-1)\ |\ [s,1,0,c-1)\}=\{[s,1,0,c-1)\ |\ [s,1,0,c-1)\}$ from Lemmas \ref{lem-czero2} and \ref{lem-azero}.
Moreover, if $a>1$, then $[s,a,0,c)=\{0,[s,a-1,0,c),[s,a,0,c-1)\ |\ [s,a,0,c-1)\}=\{[s,a,0,c-1)\ |\ [s,a,0,c-1)\}$ from Lemma \ref{lem-czero2}. Then, by induction, $[s,a,0,c)=\{\ [s,a,0,0)+\Delta(c-1)\ |\ [s,a,0,0)+\Delta(c-1)\}\ =\ a+\Delta(c)$ from Lemma \ref{lem-czero}.

So we may also assume that $b>0$.
Notice that $[s,1,b,c)=\{0,[s,0,b,c),[s,1,b,c-1)\ |\ [s,1,b-1,c), [s,1,b,c-1)\}=\{\Gamma(b+c),[s,1,b,0)+\Delta(c-1)\ |\ [s,1,b-1,0)+\Delta(c), [s,1,b,0)+\Delta(c-1)\}$ by induction and Lemmas \ref{lem-czero2} and \ref{lem-azero}.
Thus we may assume that $a\geq 2$.
Then, $[s,a,b,c)=\{0,[s,a-1,b,c),[s,a,b,c-1)\ |\ [s,a,b-1,c), [s,a,b,c-1)\}$. By induction, $[s,a,b,c)=\{0,[s,a-1,b,0)+\Delta(c), [s,a,b,0)+\Delta(c-1)\ |\ [s,a,b-1,0)+\Delta(c), [s,a,b,0)+\Delta(c-1)\}$. From Lemma \ref{lem-czero2}, $[s,a,b,c)=\{\ [s,a,b,0)+\Delta(c-1)\ |\ [s,a,b,0)+\Delta(c-1)\}=[s,a,b,0)+\Delta(c)$.
\end{proof}

\begin{theorem}\label{thm-caso1}
Let $s\geq 2$ and $c\geq 2$. Then $(s,0,0,c)$ is $0$, if $c=2$, and is $\Gamma(c)$, if $c>2$. 
\end{theorem}

\begin{proof}
Notice that $(s,0,0,2)=0$, since the first player always loses, and that $(s,0,0,3)=\{(s-1,0,0,3],[s,0,0,2)\ |\ [s,0,0,2)\}=\{\uparrow^{[s-2]}*,\uparrow\ |\ \uparrow\}=\Gamma(3)$ from Theorem \ref{thm-basic3}. 
Recall that $(s-1,0,0,c]=\uparrow^{[s-2]}*$ from Lemma \ref{lem-esq}.
Also notice that, if $c\geq 4$, then $(s,0,0,c)=\{(s-1,0,0,c],[s,0,0,c-1)\ |\ [s,0,0,c-1)\}=\{\uparrow^{[s-2]}*,\Gamma(c-1)\ |\ \Gamma(c-1)\}$ by Lemma \ref{lem-azero} and consequently $(s,0,0,c)=\{\Gamma(c-1)\ |\ \Gamma(c-1)\}=\Gamma(c)$ from Theorems \ref{thm-basic1} and \ref{thm-basic3}.
\end{proof}

\begin{theorem}\label{thm-caso2}
Let $s\geq 2$ and $a,b\geq 1$. Then $(s,a,b,0)$ is $a-b$, if $a>b$, and is $1/2^{b-a+1}$, if $a\leq b$.
\end{theorem}

\begin{proof}
Notice that $(s,a,b,0)=\{(s-1,a,b,0],[s,a-1,b,0)\ |\ [s,a,b-1,0)\}$.
Recall that $(s-1,a,b,0]=\uparrow^{[s-2]}*$ from Lemma \ref{lem-esq}.
Therefore, from Lemmas \ref{lem-czero}, we have the following cases.
If $a>b+1$, then $(s,a,b,0)=\{\uparrow^{[s-2]}*,a-b-1\ |\ a-b+1\}=a-b$.
If $a=b+1$, then $(s,a,b,0)=\{\uparrow^{[s-2]}*,1/2\ |\ 2\}=1=a-b$.
If $a=b$, then $(s,a,b,0)=\{\uparrow^{[s-2]}*,1/4\ |\ 1\}=1/2=1/2^{b-a+1}$.
If $a<b$, then $(s,a,b,0)=\{\uparrow^{[s-2]}*,1/2^{b-a+2}\ |\ 1/2^{b-a}\}=1/2^{b-a+1}$.
\end{proof}

\begin{theorem}\label{thm-caso3}
Let $s\geq 2$ and $a,c\geq 1$. Then $(s,a,0,c)=a+\Delta(c)$.
\end{theorem}

\begin{proof}
Notice that $(s,a,0,c)=\{(s-1,a,0,c],[s,a-1,0,c),[s,a,0,c-1)\ |\ [s,a,0,c-1)\}$.
Recall that $(s-1,a,0,c]=\uparrow^{[s-2]}*$ from Lemma \ref{lem-esq}.
Then, by Lemmas \ref{lem-czero}, \ref{lem-czero2} and \ref{lem-cdifz}, $(s,a,0,c)=\{\ [s,a,0,c-1)\ |\ [s,a,0,c-1)\}=\{a+\Delta(c-1)\ |\ a+\Delta(c-1)\}=a+\Delta(c)$.
\end{proof}

\begin{theorem}\label{thm-caso4}
Let $s\geq 2$ and $b,c\geq 1$. Then $(s,0,b,c)$ is $0$, if $b=c=1$, and is $\Gamma(b+c)$, otherwise.
\end{theorem}

\begin{proof}
Clearly, $(s,0,1,1)=0$, since the first player always loses.
So assume that $t=b+c\geq 3$.
Notice that $(s,0,b,c)=\{(s-1,0,b,c],[s,0,b,c-1)\ |\ [s,0,b-1,c),[s,0,b,c-1)\}$.
Recall that $(s-1,0,b,c]=\uparrow^{[s-2]}*$ from Lemma \ref{lem-esq}.
Then $(s,0,b,1)=\{\uparrow^{[s-2]}*,\Gamma(b)\ |\ \Gamma(b+1),\Gamma(b)\}=\{\Gamma(b)|\Gamma(b)\}=\Gamma(b+1)$ by Lemmas \ref{thm-basic1}, \ref{lem-czero} and \ref{lem-cdifz}. 
Moreover, if $c\geq 2$, then $(s,0,b,c)=\{\uparrow^{[s-2]}*,\Gamma(b+c-1)\ |\ \Gamma(b+c-1),\Gamma(b+c-1)\}=\{\Gamma(b+c-1)|\Gamma(b+c-1)\}=\Gamma(b+c)$ by Lemmas \ref{thm-basic1}, \ref{lem-azero} and \ref{lem-cdifz}.
\end{proof}

\begin{theorem}\label{thm-caso5}
Let $s\geq 2$ and $a,b,c\geq 1$. Then $(s,a,b,c)=[s,a,b,0)+\Delta(c)$.
\end{theorem}

\begin{proof}
Notice that $(s,a,b,c)=\{(s-1,a,b,c],[s,a-1,b,c),[s,a,b,c-1)|[s,a,b-1,c),[s,a,b,c-1)\}$.
Then, from Lemma \ref{lem-cdifz}, 
$(s,a,b,c)=\{\uparrow{[s-2]}*,[s,a-1,b,0)+\Delta(c),[s,a,b,0)+\Delta(c-1)]\ |\ [s,a,b-1,0)+\Delta(c),[s,a,b,0)+\Delta(c-1)\}$.
Therefore, from Lemma \ref{lem-czero2}, $(s,a,b,c)=\{[s,a,b,0)+\Delta(c-1)\ |\ [s,a,b,0)+\Delta(c-1)\}\ =\ [s,a,b,0)+\Delta(c)$.
\end{proof}

\section{Conclusions}

In this paper, we obtained the value of the Normal Domination Partizan game for any complete bipartite graph and any complete split graph for any possible initial coloring of its vertices. This implies from the Combinatorial Game Theory that determining the winner of this game is polynomial time solvable in graphs whose connected components are complete bipartite graphs or complete split graphs (for any possible initial coloring of its vertices).

In the following, we present the algorithm that returns the value of the game in stars, summarizing the results of Theorems \ref{thm-star1}, \ref{thm-star2}, \ref{thm-star3}, \ref{thm-star4} and \ref{thm-star5}.
Let $-G$ be the instance of the game obtained from the colored graph $G$ by swapping the colors $A$ and $B$ in the vertices.
Recall that $\Delta(k)=*$, if $k$ is odd, and $\Delta(k)=0$, if $k$ is even, and that $\Gamma(k)=(k-1)\uparrow+\Delta(k)$.

\begin{algorithm}
\caption{Value of the game in Stars}\label{alg-star}
\begin{algorithmic}[1]
\Procedure{Value-Star}{star graph $G$}    \Comment{\emph{The vertices of $G$ are colored $A$, $B$ or $C$}}
\If{$G$ has only one vertex} \Return $1$, $-1$, $*$ if the color is $A$, $B$ or $C$, resp.\EndIf
\If{$G$ has only two vertices}
    \If{color $C$ appears or there are two colors} \Return $*$\EndIf
    \State\Return $1$ or $-1$ depending on if the only color is $A$ or $B$, resp.
\EndIf
\State $a,b,c\gets$ number of leaves with color $A$, $B$ and $C$, resp.
\State Let $G'$ be obtained from $G$ by removing the leaves of color $C$
\If{the center of the star has color $A$}
    \If{$a=0$} \Return $\Gamma(b+c)$ \EndIf
    \If{$a\leq b$ and $c=0$} \Return $1/2^{b-a+1}$ \EndIf
    \If{$a>b$ and $c=0$} \Return $a-b$ \EndIf
    \State\Return $\textsc{Value-Star}(G')+\Delta(c)$
\EndIf
\If{$a<b$ or the center of the star has color $B$}
    \Return $-\textsc{Value-Star}(-G)$
\EndIf
\If{the center of the star has color $C$}
    \If{$a=b=0$} \Return $*$ or $*2$ depending on if $c$ is odd or even, resp.\EndIf
    \If{$a\geq b+2$ and $c=0$} \Return $\uparrow^{[a-b-1]}*$\EndIf
    \If{$c=0$} \Return $*2$ or $*$ depending on if $a=b$ or $a=b+1$, resp.\EndIf
    \If{$c\geq 2$ is even} \Return $\textsc{Value-Star}(G')$\Comment{\emph{from now on, $c\geq 1$ is odd}}\EndIf
    \If{$a\geq b+2$} \Return $\{0,\uparrow^{[a-b-1]}*\ |\ 0, \uparrow^{[a-b-1]}*\}$\EndIf
    \State\Return $*$ or $*2$ depending on if $a=b$ or $a=b+1$, resp.\Comment{\emph{last case: $a\in\{b,b+1\}$}}
\EndIf
\EndProcedure
\end{algorithmic}
\end{algorithm}

In the following, we present the algorithm that returns the value of the game in complete split graphs, summarizing the results of Theorems \ref{thm-clique} and \ref{thm-split}.

\begin{algorithm}
\caption{Value of the game in Complete Split Graphs}\label{alg-split}
\begin{algorithmic}[1]
\Procedure{Value-complete-Split}{complete split graph $G$}
\If{$G$ is a star} \Return $\textsc{Value-Star}(G)$\EndIf
\State Let $K$ and $S$ be the clique and the stable set of $G$ with $|K|\geq 2$
\If{$S=\emptyset$}\Comment{\emph{case of complete graphs}}
    \If{color $C$ appears or there are two colors} \Return $*$\EndIf
    \State\Return $1$ or $-1$ depending on if the only color is $A$ or $B$, resp.
\EndIf
\State Let $G'$ be obtained from $G$ by identifying all vertices of $K$ in a single vertex $u$ colored $C$
\If{all vertices of $K$ have color $A$} Recolor $u$ with color $A$\EndIf
\If{all vertices of $K$ have color $B$} Recolor $u$ with color $B$\EndIf
\State \Return $\textsc{Value-Star($G'$)}$
\EndProcedure
\end{algorithmic}
\end{algorithm}

In the following, we present the algorithm that returns the value of the game in complete bipartite graphs, summarizing the results of Theorems \ref{thm-bip1}, \ref{thm-caso1}, \ref{thm-caso2}, \ref{thm-caso3}, \ref{thm-caso4} and \ref{thm-caso5}.

\begin{algorithm}
\caption{Value of the game in Complete Bipartite Graphs}\label{alg-bip}
\begin{algorithmic}[1]
\Procedure{Value-complete-Bip}{complete bipartite graph $G$}    
\If{$G$ is a star} \Return $\textsc{Value-Star}(G)$\EndIf
\State Let $S$ and $T$ be the parts of $G$ with sizes $s=|S|\geq 2$ and $t=|T|\geq 2$
\If{all vertices have color $A$} \Return $\max\{s,t\}$\EndIf
\If{all vertices have color $B$} \Return $-\max\{s,t\}$\EndIf
\If{color $C$ or at least two colors appear in both $S$ and $T$} \Return $0$\EndIf
\If{all vertices of $S$ are colored $A$ and all of $T$ are colored $B$} \Return $0$\EndIf
\If{all vertices of $S$ are colored $B$ and all of $T$ are colored $A$} \Return $0$\EndIf
\State Swap $S$ and $T$ if $S$ has color $C$ or at least two colors.\Comment{\emph{That is, from now on, all vertices of $S$ have the same color $A$ or $B$ and $T$ contains the color $C$ or at least two colors}}
\If{all vertices of $S$ are colored $B$} \Return $-\textsc{Value-complete-Bip}(-G)$\EndIf
\State $a,b,c\gets$ number of vertices of $T$ colored $A$, $B$ and $C$, resp.\Comment{\emph{from now on, all vertices of $S$ have color $A$}}
\If{$a>b$} \Return $a-b+\Delta(c)$ \EndIf
\If{$b\geq a\geq 1$} \Return $1/2^{b-a+1}+\Delta(c)$ \EndIf
\If{$a=0$} \Return $0$ or $\Gamma(b+c)$ depending on if $b+c=2$ or $b+c>2$\EndIf
\EndProcedure
\end{algorithmic}
\end{algorithm}

\bibliographystyle{plain}
\bibliography{refs}

\end{document}